\documentclass{amsart}
        
\usepackage[utf8]{inputenc}
\usepackage{amsmath,amssymb,amsthm,units, stmaryrd,stackrel,relsize,bm}
\usepackage{enumitem}
\usepackage{mathdots}
\usepackage{bussproofs}
\usepackage{tikz}
\usetikzlibrary{calc,arrows, decorations.pathmorphing, matrix,fit, cd}
\usepackage{float}
\usepackage{soul} 

\newtheorem{theorem}{Theorem}

\newtheorem{lemma}[theorem]{Lemma}

\newtheorem{claim}[theorem]{Claim}
\newtheorem{corollary}[theorem]{Corollary}
\newtheorem{proposition}[theorem]{Proposition}

\theoremstyle{definition}
\newtheorem{definition}[theorem]{Definition}

\theoremstyle{remark}
\newtheorem{remark}[theorem]{Remark}

\usepackage[
    colorlinks=true, 
    citecolor=red,
    linkcolor=blue,
    urlcolor=blue]{hyperref}
\newcommand\kp{{\mathsf{KP}}}
\newcommand\kpi{{\mathsf{KPi}}}

\newcommand\WO{{\mathsf{WO}}}
\newcommand\FO{{\mathsf{FO}}}
\newcommand\LO{{\mathsf{LO}}}
\newcommand\Nat{{\mathsf{Nat}}}

\newcommand\supp{{\text{supp}}}

\newcommand\ca{{\mathsf{CA_0}}}

\newcommand\aca{{\mathsf{ACA_0}}}
\newcommand\rca{{\mathsf{RCA_0}}}

\newcommand\id{\mathsf{id}}
\makeatletter
\def\Ddots{\mathinner{\mkern1mu\raise\p@
\vbox{\kern7\p@\hbox{.}}\mkern2mu
\raise4\p@\hbox{.}\mkern2mu\raise7\p@\hbox{.}\mkern1mu}}
\makeatother
\newsavebox{\pullback}
\sbox\pullback{%
\begin{tikzpicture}%
\draw (0,0) -- (1ex,0ex);%
\draw (1ex,0ex) -- (1ex,1ex);%
\end{tikzpicture}}

\begin{document}
\title{Functorial Fast-Growing Hierarchies}
\author{J. P. Aguilera}
\address{Institute of Discrete Mathematics and Geometry, Vienna University of Technology. Wiedner Hauptstra{\ss}e 8--10, 1040 Vienna, Austria \textit{and} Department of Mathematics, University of Ghent. Krijgslaan 281-S8, B9000 Ghent, Belgium}
\email{aguilera@logic.at}
\author{F. Pakhomov}
\address{Department of Mathematics, University of Ghent. Krijgslaan 281-S8, B9000 Ghent, Belgium \textit{and} Steklov Mathematical Institute of the Russian Academy of Sciences. Ulitsa Gubkina 8, Moscow 117966, Russia.}
\email{fedor.pakhomov@ugent.be}
\author{A. Weiermann}
\address{Department of Mathematics, University of Ghent. Krijgslaan 281-S8, B9000 Ghent, Belgium.}
\email{andreas.weiermann@ugent.be}
%\begin{document}

\begin{abstract}
Fast-growing hierarchies are sequences of functions obtained through various processes similar to the ones that yield multiplication from addition, exponentiation from multiplication, etc. 
We observe that fast-growing hierarchies can be naturally extended to functors on the categories of natural numbers and of linear orders. 
We show that the categorical extensions of binary fast-growing hierarchies to ordinals are isomorphic to denotation systems given by ordinal collapsing functions,  thus establishing a connection between two fundamental concepts in Proof Theory.

Using this fact, we obtain a restatement of the subsystem  $\Pi^1_1$-$\ca$ of analysis as a higher-type wellordering principle.
\end{abstract}
\date{\today $\,$ (compiled)}
\clearpage
\maketitle
\numberwithin{equation}{section}
\setcounter{tocdepth}{1}
\tableofcontents

\section{Introduction}
A \emph{fast growing hierarchy} is a collection of functions $B_\alpha:\mathbb{N}\to\mathbb{N}$, where $\alpha$ ranges over some set of ordinal numbers, and each $B_\alpha$ eventually dominates $B_\beta$, for all $\beta<\alpha$. 
One is mainly interested in fast-growing hierarchies obtained through a given recursive construction which usually generalizes the process that produces multiplication as iterated addition, exponentiation as iterated multiplication, etc.

Fast-growing hierarchies can be defined in various ways, though all reasonable definitions that are known yield sequences of functions that agree at ordinals $\alpha$ which are sufficiently closed. 
Examples of these are the \emph{Hardy hierarchy} (named after G. H. Hardy), the \emph{iterative fast growing hierarchy}, and the \emph{binary fast growing hierarchy}. Let us focus on the latter one, for the sake of definiteness. 
It is defined by transfinite induction according to the following construction:
\begin{align*}
B_0(n) &= n+1&\\
B_{\alpha+1}(n) &= B_\alpha\circ B_\alpha(n)&\\
B_\lambda(n) &= B_{\lambda[n]}(n), &\text{ at limit stages,}
\end{align*}
where for each limit ordinal $\lambda$, $\{\lambda[n]:n\in\mathbb{N}\}$ is a sequence which converges to $\lambda$, fixed in advance. 
The precise functions defined depend on the sequences chosen, but for natural choices of sequences, the functions behave very regularly. The ordinal $\varepsilon_0$ is defined by:
\[\varepsilon_0 = \sup\{\omega,\omega^\omega,\omega^{\omega^\omega},\hdots\},\]
$B_{\varepsilon_0}$ is a total recursive function whose totality cannot be proved in Peano Arithmetic (or, equivalently, Arithmetical Comprehension, $\aca$).
The foundational significance of fast-growing functions was observed by Ackermann and, later, by Kreisel \cite{Kr52}.
\begin{theorem}[Kreisel 1952, essentially]\label{KreiselsTheorem}
There is a canonical wellordering $W$ of $\mathbb{N}$ of length $\varepsilon_0$ such that the following theories prove the same $\Pi^0_2$ theorems:
\begin{enumerate}
\item Arithmetical Comprehension;
\item Primitive Recursive Arithmetic + $\{B_\alpha$ is total  $:\alpha\in W\}$.
\end{enumerate}
\end{theorem}
Kreisel's theorem is a computational analysis of arithmetic, in the sense that it gives an explicit characterization of the recursive functions which are provably total.

%The totality of a fixed recursive function can itself be expressed as a $\Pi^0_2$ sentence. Thus, the assertion that a function such as $B_{\varepsilon_0}$ is total strengthens $\aca$  by adding computational strength, but it does not imply the existence of any  new infinite objects.  

Kreisel's result is of great foundational significance and has led to a great deal of metamathematical and combinatorial results in the context of arithmetic, as well as generalizations to stronger systems. For instance, one can prove a version of Kreisel's theorem for the subsystem $\Pi^1_1$-$\ca$ of analysis where $B_{\varepsilon_0}$ is replaced by $B_{\psi_0(\aleph_\omega)}$ (see Buchholz \cite{Bu87}).
Here, $\psi_0(\aleph_\omega)$ is the \emph{Takeuti ordinal}, a recursive ordinal that is most easily described by \emph{ordinal collapsing functions} for the uncountable cardinals $\aleph_n$, where $n\in\mathbb{N}$. 
$\Pi^1_1$-$\ca$ is a historically and mathematically important theory, and  is the strongest of the ``big five'' subsystems of Second-Order Arithmetic commonly studied in Reverse Mathematics. Ordinal collapsing functions were introduced by Bachmann \cite{Ba50} and are a powerful method in proof theory involved in the description of large recursive ordinals, such as the \emph{Bachmann-Howard} ordinal $\psi(\varepsilon_{\Omega+1})$, the Takeuti ordinal, and other larger numbers such as the proof-theoretic ordinals of $\kpi$ and $\kp$ + $\Pi_3$-reflection. These ordinals involve collapsing functions for inaccessible and weakly compact cardinals, respectively; see Rathjen \cite{Ra99} for an overview.

The first observation we make is that fast-growing hierarchies can be regarded as functors on the category of natural numbers with strictly increasing functions as morphisms. This fact was perhaps underpinning the main result of \cite{AFRW}. Let us say more about this observation.

For our purposes, it will be convenient to shift our focus to the norm-based presentation of fast-growing hierarchies.
This presentation goes back to \cite{BCW94} and, for our purposes, is described as follows: suppose that we are given a denotation system $D$ for ordinals.
The definition of ``denotation system'' is recalled in the next section, where these are defined as  a particular kind of \emph{dilator.} The reader unfamiliar with denotation systems might temporarily  think of $D$ as a system for representing ordinals in some way using natural numbers as parameters. For example, we may think of representing ordinals ${<}\omega^\omega$ as sums of powers of $\omega$; e.g.,
\[\omega^7 + \omega^3+ 21.\]
The \emph{norm} $N\alpha$ of an ordinal $\alpha$ is the strict supremum of the natural number parameters which occur in the denotation.
We define
\[B_\alpha(n) = \sup\{B_\beta\circ B_\beta(n): \beta<\alpha\wedge N\beta \leq n\} \cup \{n+1\}.\]

Natural numbers can be represented by repeated applications of functions in the hierarchy:
\begin{equation} \label{eqIntroNormalForm}
m = B_{\alpha_k} B_{\alpha_{k-1}}\dots B_{\alpha_1} (n).
\end{equation}
If $n$ is fixed and some constraints are imposed on the values $n_1, \hdots, n_k$, then the resulting representation is unique. 
We define
\[B_D(n) = B_{D(\omega)}(n) = \sup\{B_\beta\circ B_\beta(n): N\beta \leq n\} \cup \{n+1\}.\]
Making use of the representation of natural numbers in terms of iterations of the fast-growing hierarchy, we can regard $B_D$ as a functor on the category of natural numbers with strictly increasing functions as morphisms. This is proved in \S \ref{SectFunctor}.

 From the fact that $B_D$ is a functor on the category of natural numbers, it follows that $B_D$ can be uniquely extended to a finitary functor on the category of linear orders. In particular, expressions of the form \eqref{eqIntroNormalForm} can be given transfinite parameters and construed as denotations for elements of some linear order. The main insight in \S \ref{SectCollapse} is that (up to isomorphism) these denotations are obtained from ordinal collapsing functions. Thus, we have an object $B_D$ which can be simultaneously viewed from the point of view of two fundamental concepts in Proof Theory: fast-growing functions and ordinal collapses.

In \S \ref{SectTheorem}, we prove the main theorem. We first show that the ordinal collapses given by $B_D$ are wellfoundedness-preserving, so that if $A$ is a wellorder, then so too is $B_D(A)$. In fact, $B_D$ is a functor on the category of ordinal numbers, and indeed a \emph{dilator}. The proof of this fact requires the use of $\Pi^1_1$-$\ca$, and this is unavoidable: the fact that $B_D$ is wellfoundedness-preserving in turn implies $\Pi^1_1$-$\ca$. This way we obtain a restatement of $\Pi^1_1$-$\ca$ in terms of fast-growing hierarchies and abstract ordinal collapses. We state the theorem; for all undefined terms, we refer the reader to the following section.

\begin{theorem}\label{MainIntro}
The following are equivalent over $\aca$:
\begin{enumerate}
\item  $\Pi^1_1$-$\ca$;
\item for every weakly finite dilator $D$, $B_D$ is a weakly finite dilator;
\item  for every dilator $D$, $B_D$ is a dilator.
\end{enumerate}
\end{theorem}

Theorem \ref{MainIntro} fits within the family of results known as \emph{wellordering principles}. Most commonly, these are principles of the form ``if $X$ is a wellorder, then $f(X)$ is a wellorder,'' where $f$ is some transformation. 
These principles have been extensively studied in Reverse Mathematics and have led to reformulations of several subsystems of analysis. Some examples are the works of Girard \cite{Gi87},  Friedman (unpublished), Hirst \cite{Hi94}, Afshari-Rathjen \cite{AfRa09}
, Marcone-Montalb\'an \cite{MaMo11}, Rathjen-Valencia Vizca\'ino \cite{RaVV15}, Thomson-Rathjen \cite{ThRa20}, and \cite{RaWe11}. Higher-type wellordering principles were used by Girard \cite{GiInf} and Freund \cite{Fr19} to characterize $\Pi^1_1$-$\ca$. Freund's theorem is an ingredient of our proof of Theorem \ref{MainIntro}. Freund's principle maps dilators to wellorderings, while ours and Girard's map dilators to dilators. Girard's characterization of  $\Pi^1_1$-$\ca$ asserts the totality of the functor $\Lambda$, which itself is reminiscent of fundamental-sequence--based fast-growing hierarchies. Wellordering principles for theories such as $\Pi^1_1$-$\ca$ which are not $\Pi^1_2$-axiomatizable require the use of higher-type objects such as dilators. 

Although throughout this article we focus on the binary fast-growing hierarchy $B_D$, we mention that the statement of Theorem \ref{MainIntro} applies to many of the other usual fast-growing hierarchies as well. The use of $B_D$ is merely for convenience and to simplify the computations involved.

\subsection*{Acknowledgements}
J.P.A. was partially supported by FWF grants I4513N and ESP-3N and FWO grant 0E3017319; F.P. was supported by FWO grant G0F8421N; A.W. was partially supported by FWO grants G0E2121N and G030620N.

\section{Preliminaries}
We deal with the category $\LO$. The objects of $\LO$ are linear orders whose domains are subsets of $\mathbb{N}$. The morphisms are strictly increasing maps between this orders. 
%As usual both the orders and the maps are represented by individual object. 
We shall work within the subsystem $\aca$ of second-order arithmetic. This is the system in the language of second-order arithmetic (which contains sorts for natural numbers and sets of natural numbers) whose principal axioms are the induction axiom and the schema asserting that every arithmetically definable set exists.
When it leads to no confusion, we shall leave the precise formalization of statements to the reader and work informally. 

In the language of second-order arithmetic one cannot generally speak of functors from $\LO$ to $\LO$; thus, we shall restrict to dealing with \emph{finitary} functors, i.e., functors preserving co-limits of upward directed diagrams (direct limits in the model-theoretic sense). These can be coded by sets of natural numbers: 

We consider $\Nat$, the full subcategory of $\LO$ where objects are natural numbers $n$, which we identify with the  orders $([0,n),<)$. 
Since both $\mathsf{Ob}_\Nat$ and $\mathsf{Mor}_\Nat$ are countable, each functor from $\Nat$ to $\LO$ can be coded by a set of natural numbers. 
Each functor $D\colon \Nat\to \LO$ can be uniquely extended to a finitary functor $D\colon \LO\to \LO$ as follows: 

Below, we write $A'\subseteq_{\mathsf{fin}} A$ if $A'$ is a finite subset (or substructure) of $A$.
Given a linear order $A$, we consider an upward directed diagram  $H_A$ consisting of all finite suborders $A'\subseteq_{\mathsf{fin}}A$ and all morphisms $\id_{A'\to A''}$, for $A'\subseteq_{\mathsf{fin}}A''\subseteq_{\mathsf{fin}}A$. 
Let $H'$ be the naturally isomorphic diagram in which all objects are orders in $\mathsf{Ob}_\Nat$. 
Applying $D$ to all morphisms and objects in $H'_A$, we obtain a diagram $D[H'_A]$ and define 
\[D(A) = \varinjlim D[H'],\]
where 
$\injlim D[H']$ denotes the co-limit of the diagram.
The value $D(f)$, for an $\LO$-morphism $f\colon A\to C$, is recovered in the natural way.

A coded finitary functor $D$ on $\LO$ is called a pre-dilator if it preserves pullbacks. Clearly, $D$ preserves pullbacks if and only if its restriction to $\Nat$ preserves pullbacks.

It will be convenient for our purposes to work with inclusion-preserving (or \emph{$\subseteq$-preserving}) functors. These are the functors $D$ such that for any order $A$ and any suborder $B$ of $A$,
\begin{enumerate}
\item $D(B)$ is a suborder of $D(A)$, and furthermore \item $D(\id_{B\to A})=\mathsf{id}_{D(B)\to D(A)}$, where $\id_{B\to A}\colon B\to A$, denotes the inclusion map given by $\id_{B\to A}(x)=x$.
\end{enumerate} 
Finitary $\subseteq$-preserving functors were called $\omega$-local functors by S.~Feferman \cite{Fe72}.\footnote{Feferman \cite{Fe72} is an early reference for functorial ordinal notation systems.}
For $\subseteq$-preserving functors there are convenient reformulations of the conditions of being a finitary functor and a pre-dilator. Namely, a $\subseteq$-preserving functor $D$ is finitary if and only if for any order $A$, we have 
\[D(A)=\bigcup\limits_{A'\subseteq_{\mathsf{fin}}A}D(A').\] 
An  $\subseteq$-preserving finitary functor $D$ is a pre-dilator iff for any orer $A$ and its suborders $B,C\subseteq A$ we have $D(B\cap C)=D(B)\cap D(C)$.

Working in $\aca$, we code $\subseteq$-preserving finitary functors as follows. 
We consider the full subcategory $\FO$ of $\LO$, where objects are finite linear orders. 
As in the case of $\Nat$, both  $\mathsf{Ob}_\FO$ and $\mathsf{Mor}_\FO$ are countable, so $\subseteq$-preserving functors $D$ from $\FO$ to $\LO$ can be naturally each by a set of natural numbers. 
We extend any such $D$ to a functor with domain $\LO$ by putting 
$$D(A)=\bigcup\limits_{A'\subseteq_{\mathsf{fin}}A}D(A')\;\;\;\text{ and }\;\;\;D(f)=\bigcup\limits_{A'\subseteq_{\mathsf{fin}}A}D(f\upharpoonright A').$$ 
Here for $f\colon A\to C$ and $A'\subseteq A$, $f\upharpoonright A'\colon A'\to C$ is the  restriction of $f$.

A pre-dilator is called a dilator if it maps well-orders to well-orders. Dilators were introduced by Girard \cite{Gi81}, to whom we refer the reader for further background. The fundamental theorem of dilators states that every pre-dilator $D$ is naturally isomorphic to a \emph{denotation system}, a special kind of $\subseteq$-preserving predilators.

Each denotation system $D$ consists of a set of terms $t(x_1, \hdots, x_n)$ and comparison rules establishing, for each pair of terms $t(\vec x)$ and $s(\vec y)$, which one is greater, depending on the relative ordering of the constants $\vec x$ and $\vec y$. Note that we do not allow $t(\vec x) = s(\vec y)$ unless $t(\vec x)$ and $s(\vec y)$ are syntactically the same term.
This allows us to define a binary relation $D(A)$ for each linear order $A$ as follows: the domain of $D(A)$ is the set of all expressions $t(a_1,\ldots,a_n)$, where $a_1>\ldots>a_n$ and $t(x_1,\ldots,x_n)$ is a $D$-term. The relative ordering of elements of the domain of $D(A)$ is determined by the comparison rules for each pair of terms.

Given a strictly increasing map $f\colon A\to C$ between finite orderings $A$ and $C$, we define 
\begin{align*}
D(f)\colon D(A) &\to D(C)\\
D(f)\colon t(a_1,\ldots,a_n)&\mapsto t(f(a_1),\ldots,f(a_n))
\end{align*}

We shall sometimes identify the values of terms $t(x_1, \hdots, x_n)$ with the ordinals they denote.
If $\alpha =t(x_1, \hdots, x_n)$, we call the set $\{x_1, \hdots, x_n\}$ the \emph{support} of $\alpha$ and denote it by $\supp(\alpha)$. Given ordinals $\alpha,\beta$, we write
\[\supp(\alpha) < \beta \text{ if $x_i<\beta$ for all $x_i \in \supp(\alpha)$}.\]

%Here, $D(\alpha)$ corresponds to the order type of all terms representable by a set of constants $x_i$ wellordered in order type $\alpha$. %Finitary functors on $\WO$ also correspond to denotation systems for ordinal numbers, the difference being that for terms $t(\vec x)$ and $s(\vec y)$, the comparison rules allow setting $t(\vec x) = s(\vec y)$, i.e., denotations are not unique.
%Note that denotation systems are automatically $\subseteq$-preserving.

%We will often refer to terms $t(x_1, \hdots, x_n)$ as \emph{ordinals}. 
%We will mainly work with dilators as denotation systems for ordinal numbers, as this will be more convenient for our purposes. Note that for $\alpha\in D(\omega)$, $N\alpha$ is the least $k$ such that $\alpha$ is the value of a term $t(x_1, \hdots, x_n)$ with $x_1, \hdots, x_n  < k$.

%The notion of support also could be naturally extended to $\subseteq$-normal finitary functors $F\colon\LO\to\LO$. For $\alpha\in F(A)$ we put 
%$$\supp(\alpha)

\subsection{Conventions} Whenever we deal with a linear order $L$, we will write $<_L$ to indicate the ordering. If no confusion arises, we may occasionally omit reference to $L$ and simply write $<$. %Similarly, if $L$ is of the form $D(A)$ for some functor $D$ and some linear order $A$, we may also write $<_D$.

When dealing with functions $f(x)$, we will sometimes omit brackets and simply write $fx$. The purpose of this is to avoid cluttering when dealing with nested application of functions.

\section{The functoriality of fast-growing hierarchies}
\label{SectFunctor}
In this section, we  fix a dilator $D$ and investigate the functorial structure of $B_D$. All the constructions in this section are formalizable in $\aca$. Note that the norms $N\alpha$ -- and hence $B$-hierarchy on $D(\omega)$ -- are unaffected if we replace $D$ by a naturally isomorphic denotation system. Hence, in order to simplify notation, we will assume that $D$ is a denotation system (in fact we will only use that it is $\subseteq$-preserving). 
If so, $N \alpha= \min \{n\mid \alpha \in D(n)\}$. %The fact that we just consider $\subseteq$-preserving dilators isn't a limitation since any dilator $D$ is naturally isomorphic to a $\subseteq$-preserving one.

\begin{definition}
Let $n$ be a natural number and $\alpha\in D(\omega)\cup \{D(\omega)\}$. An \emph{$(n,\alpha)$-normal form term} is an expression of one of the following forms:
\begin{enumerate}
\item $m$, where $m<n$ is a natural number; or
\item $B_{\alpha_k}\dots B_{\alpha_1}(n)$, where 
\begin{enumerate}
\item $\alpha_k <_{D(\omega)} \dots <_{D(\omega)} \alpha_1<\alpha$; and
\item $N\alpha_{i+1} \leq B_{\alpha_i}\dots B_{\alpha_1}(n)$ for all $i$ with $0\leq i<k$.
\end{enumerate}
\end{enumerate}
\end{definition}

Normal form terms denote natural numbers, obtained simply by evaluating the functions. If $m = B_{\alpha_k}\dots B_{\alpha_1}(n)$ is as above, we may omit reference to $\alpha$ and call this expression the \emph{$n$-normal form} of $m$; if so, we may write 
\[m \stackrel{NF}{=} B_{\alpha_{k}}\dots B_{\alpha_1}(n).\]

For finite sequences of ordinals $(\alpha_1,\ldots,\alpha_k)\in (D(\omega))^{<\omega}$ the lexicographical comparision $<_{\mathsf{lex}}$ is defined in the usual manner: by putting 
\[(\alpha_1,\ldots,\alpha_k)<_{\mathsf{lex}} (\beta_1,\ldots,\beta_l)\] 
if there exists $0\le n\le \max(k,l)$ such that $\alpha_i=\beta_i$, for $1\le i\le n$ and either $n=k<l$ or $n<\max(k,l)$ and $\alpha_{n+1}<_{D(\omega)}\beta_{n+1}$.

\begin{lemma}[$\aca$]\label{normal_forms} 
Suppose $\alpha\in D(\omega)\cup \{D(\omega)\}$ and $n\in\mathbb{N}$. Then, 
\begin{enumerate}
\item each number smaller than $B_{\alpha}(n)$ is the value of a unique $(n,\alpha)$-normal form; and
\item  for any two $(n,\alpha)$-normal forms $B_{\alpha_1}\ldots B_{\alpha_k} (n)$ and $B_{\beta_l}\ldots B_{\beta_1} (n)$, we have
  $$B_{\alpha_1}\ldots B_{\alpha_k} (n)< B_{\beta_l}\ldots B_{\beta_1} (n)\iff (\alpha_1,\ldots,\alpha_k)<_{\mathsf{lex}}(\beta_1,\ldots,\beta_l).$$
\end{enumerate}
\end{lemma}
\begin{proof}
  We prove the lemma by transfinite induction on $\alpha$. If there are no $\alpha'<_{D(\omega)}\alpha$ with $N\alpha'\le n$, then $B_\alpha(n)=n+1$ and the only $(n,\alpha)$-normal forms are constants $\le n$ and hence the lemma holds.

Otherwise, choose $\alpha'<_{D(\omega)}\alpha$ such that $N\alpha'\leq n$ and $B_\alpha(n)=B_{\alpha'}(B_{\alpha'}(n))$. Let $s=B_{\alpha'}(n)$. By induction hypothesis,
  \begin{enumerate}
  \item any number $<B_{\alpha'}(n)$ is the value of a unique $(s,\alpha')$-normal form;
  \item any number $<s$ is the value of a unique $(n,\alpha')$-normal form.
  \end{enumerate}
Combining these two facts we observe that each number $<B_{\alpha}(n)$ is the value of the unique term of one of the following three forms:
  \begin{enumerate}
  \item $m$, for $m<n$;
  \item \label{nf_2} $B_{\alpha_k}\dots B_{\alpha_1}(n)$, where $\alpha_k <_{D(\omega)} \dots <_{D(\omega)} \alpha_1<_{D(\omega)}\alpha'$ and $N\alpha_{i+1} \leq B_{\alpha_i}\dots B_{\alpha_1}(n)$ for all $i$ with $0\leq i<k$;
    \item \label{nf_3} $B_{\alpha_k}\dots B_{\alpha_1}(B_{\alpha'}(n))$, where  $\alpha_k <_{D(\omega)} \dots <_{D(\omega)} \alpha_1<_{D(\omega)}\alpha'$ and $N\alpha_{i+1} \leq B_{\alpha_i}\dots B_{\alpha_1}(B_{\alpha'}(n))$ for all $i$ with $0\leq i<k$.
    \end{enumerate}
Since $\alpha'<_{D(\omega)}\alpha$ and $N\alpha' \leq n$, it follows that any number $<B_\alpha(n)$ is indeed the value of an $(n,\alpha)$-normal form. To prove that $(n,\alpha)$-normal forms are compared as prescribed, we simply need to consider the cases of the forms \eqref{nf_2} and \eqref{nf_3}. 
If both normal forms are of the same form, then we get the comparison property directly by induction hypothesis.
The fact that the comparison property holds for terms of different forms follows from the fact that any term of the form \eqref{nf_2} has smaller value than any term of  the form \eqref{nf_3}. This is because, by construction, every term of the form \eqref{nf_2} has value $<s$ and every term of the form \eqref{nf_3} has value $\ge s$.\end{proof}

\begin{definition}\label{DefinitionBDf}
Let $f:n \to n'$ be a strictly increasing function. We define the function $B_D(f):B_D(n) \to B_D(n')$ to be the unique strictly increasing function such that
\begin{enumerate}
\item\label{funct_def_1} if $m<n$, then $B_D(f)(m) = f(m)$;
\item\label{funct_def_2} if $n\le m<B_D(n)$ and $m \stackrel{NF}{=} B_{\alpha_{k}}\dots B_{\alpha_1}(n)$, then
\[B_D(f)(m)=B_{\alpha'_{k}}\dots B_{\alpha'_1}(n'),\]
where $\alpha'_i=D(B_D(f))(\alpha_i)$ for each $i$.
\end{enumerate}
\end{definition}

\begin{lemma}[$\aca$]
Let $f:n \to n'$ be a strictly increasing function. Then, there is a unique strictly increasing function satisfying Definition \ref{DefinitionBDf}.
\end{lemma}
\proof 
By induction on $s\leq B_D(n)$,  we define a sequence of strictly increasing maps
\[g_s: s \to B_D(n')\]
such that $g_{s+1}$ extends $g_s$ for each $s$. 
The map $g_0\colon 0 \to B_D(n')$ is simply the empty map.

Assume that we already have defined maps $g_0,\ldots,g_s$ and established that they are an increasing family of strictly increasing functions. We define the map \[g_{s+1}\colon (s+1)\to B_D(n')\] 
as follows:
\begin{enumerate}
\item if $m<n$, then we put $g_{m+1}(m)=m$;
\item \label{funct_adef_2} if $n \leq m \stackrel{NF}{=} B_{\alpha_{k}}\dots B_{\alpha_1}(n)$, then we put $g_{s+1}(m)=B_{\alpha'_{k}}\dots B_{\alpha'_1}(n')$, where  $\alpha'_{i}=D(g_s)(\alpha_{i})$ for each $i$.
\end{enumerate}
Note that in the second clause the values $D(g_s)(\alpha_i)$ are well-defined. This is because, by the condition on $n$-normal forms, we have 
\begin{equation}\label{eqNalphai<s}
N\alpha_i\le B_{\alpha_{i-1}}\dots B_{\alpha_1}(n')<m\le s,
\end{equation}
and hence $\alpha_i\in D(s)$. 

Let us check that the expression $g_{s+1}(m)$ in \eqref{funct_adef_2} is an $n$-normal form. The fact that $\alpha'_k<_{D(\omega)}\ldots<_{D(\omega)}\alpha'_1$ simply follows from the functoriality of $D$. We show that $N \alpha_i'\le B_{\alpha'_{i-1}}\dots B_{\alpha'_1}(n')$. By definition,
$\alpha'_i=D(g_s)(\alpha_i)$.
Thus, $\alpha'_i$ is the value a $D$-term where all the constants result from shifting constants ${\leq}N\alpha_i$ according to $g_s$. Hence, $\alpha'_i$ is the value of a $D$-term with constants $\leq{g_s(N\alpha_i)}$. In other words,
 $D(g_s)(\alpha_i) \in D(g_s(N\alpha_i))$. By the definition of $N$, it follows that
$$N\alpha'_i\le g_s(N\alpha_i)\le g_s(B_{\alpha_{i-1}}\dots B_{\alpha_1}(n))=B_{\alpha'_{i-1}}\dots B_{\alpha'_1}(n'),$$
as desired.
Hence all the values produced in \eqref{funct_adef_2} are indeed normal forms. 

Now, the comparision algorithm for $n$-normal forms provided by Lemma \ref{normal_forms} implies that $g_{s+1}$ is strictly increasing. If $s=0$, then $g_{s+1}$ extends $g_{s}$ simply because $g_0$ is the empty map. If $s>0$ then we immediately obtain that $g_{s+1}$ extends $g_s$ from their definitions and the fact that $g_{s}$ extends $g_{s-1}$.

Consider 
\[g\colon B_D(n)\to B_D(n')\]
given by
\[g=\bigcup\limits_{m\le B_D(n)}g_m.\]
From the definition, it is clear that $g$ satisfies properties \eqref{funct_def_1} and \eqref{funct_def_2} of the definition of $B_D(f)$. By a straightforward induction on $m$ we show that if 
\[g'\colon B_D(n)\to B_D(n')\] 
is any other function satisfying properties \eqref{funct_def_1} and \eqref{funct_def_2} of the definition of $B_D(f)$, then $g(m)=g'(m)$. Therefore, the definition of $B_D(f)$ indeed defines a unique strictly increasing function.
\endproof

\begin{corollary} \label{CorollaryFDFunctorNatural}
$B_D$ is a functor on the category $\mathsf{Nat}$.
\end{corollary}

%As a consequence of Corollary \ref{CorollaryFDFunctorNatural}, $B_D$ can be uniquely extended to a functor on the category of all linear orders in a way such that $B_D$ preserves directed co-limits.

Commenting on an earlier draft of this article, Wainer has brought to our attention an earlier result of his \cite{Wa99} in which he establishes functoriality for a version of the binary fast-growing hierarchy in the context of tree-ordinals.

\section{Fast-growing hierarchies as ordinal collapses}\label{SectCollapse}
In this section, we give an alternate definition of $B_D$ in terms of formal ``ordinal collapses'' which will directly result in a $\subseteq$-preserving dilator. 
We will show that the new construction, as a functor, is naturally isomorphic to the one constructed in the previous section. 
It will be clear from the construction that this definition coincides with the previous one when $A$ is a natural number with the usual ordering. 
As a consequence of this and Corollary \ref{CorollaryFDFunctorNatural}, it follows that the new definition coincides with the previous one for arbitrary linear orderings $A$. Again, all the constructions will be formalizable in $\mathsf{ACA}_0$.

Recall that for a linear order $A$, $2^A$ is the linear order consisting of formal base-$2$ Cantor normal forms \[2^{a_1}+\ldots+2^{a_n}\] 
with $a_1>_A\ldots>_Aa_n$. These formal sums are compared in the natural way, i.e., 
$2^{a_1}+\ldots+2^{a_n}<_{2^A}2^{b_1}+\ldots+2^{b_m}$ if $(a_1,\ldots,a_n)<_{\mathsf{lex}}(b_1,\ldots,b_m)$. Here, the empty sum is allowed and identified with the term $0$.
This construction naturally extends to a dilator: given a strictly increasing  function $f\colon A\to B$, we define
\begin{align*}
2^f\colon 2^A &\to 2^B\\
2^f\colon 2^{a_1}+\ldots+2^{a_n}&\mapsto 2^{f(a_1)}+\ldots+2^{f(a_n)}.
\end{align*}

Below, if $A$ is a linear ordering and $a \in A$, we denote by $A\upharpoonright a$ the initial segment $\{x\in A: x<a\}$, viewed as a linear order.
\begin{definition}
Let $A$ be a linear ordering. $B_D(A)$ is defined to be the \mbox{(inclusion-)least} linear ordering $C$ such that the following hold:
\begin{enumerate}
\item $C$ contains the term $a^\star$, for any $a\in A$.
\item $C$ contains the term $\psi(0)$.
\item Suppose that:
\begin{enumerate}
\item $\alpha_1,\ldots,\alpha_k,\alpha_{k+1}\in D(C)$,
\item $\psi(2^{\alpha_1}+\ldots+2^{\alpha_k})\in C$,
\item $\alpha_{k+1}<_{D(C)}\alpha_k$, and 
\item $\alpha_{k+1}\in D(C\upharpoonright\psi(2^{\alpha_1}+\ldots+2^{\alpha_k}))$;
\end{enumerate}
 then   $\psi(2^{\alpha_1}+\ldots+2^{\alpha_k}+2^{\alpha_{k+1}})\in C$.
\item if $a<_Ab$, then $a^{\star}<_Cb^\star$.
\item $a^\star<_C\psi(t)$, for any  $a^\star,\psi(t)\in C$.
\item if $\psi(t), \psi(u)\in C$, $t,u\in 2^{D(C)}$, and $t<_{2^{D(C)}}u$, then  $\psi(t)<_C\psi(u)$.
\end{enumerate}
\end{definition}

\begin{lemma}[$\aca$]
Suppose $D$ is a dilator and $A$ is a linear ordering. Then, $B_D(A)$ exists.
\end{lemma}
\proof
The idea is to construct $B_D(A)$ as the limit of an inductive definition of length $\omega$.
We construct linear orders 
\[B_{D,0}(A)\subseteq B_{D,1}(A)\subseteq \ldots \subseteq B_{D,n}(A)\subseteq\ldots\]
Recursively,we denote by $B_{D,<n}(A)$ the already defined linear order $\bigcup\limits_{m<n}B_{D,m}(A)$ and define $B_{D,n}(A)$ to contain the following terms:
\begin{enumerate}  
\item $a^\star$, for $a\in A$;
\item $\psi(0)$;
\item $\psi(2^{\alpha_1}+\ldots+2^{\alpha_k}+2^{\alpha_{k+1}})$, whenever $\alpha_1,\ldots,\alpha_k,\alpha_{k+1}\in D(B_{D,<n}(A))$ are such that:
\begin{enumerate}
\item $\psi(2^{\alpha_1}+\ldots+2^{\alpha_k})\in B_{D,<n}(A)$, 
\item $\alpha_{k+1}<_{D(B_{D,<n}(A))}\alpha_k$, and 
\item $\alpha_{k+1}\in D( B_{D,<n}(A){\upharpoonright} \psi(2^{\alpha_1}+\ldots+2^{\alpha_k}))$.
\end{enumerate}
\end{enumerate}
The order on these terms is defined in the same way as above, i.e., $a^\star<_{B_{D,<n}(A)}b^\star$ if and only if $a<_Ab$; we always have $a^\star<_{B_{D,<n}(A)}\psi(t)$; and $\psi(t)<_{B_{D,n}(A)}\psi(u)$ if and only if $t<_{2^{D(B)}}u$. It is straightforward to check that, provably in $\mathsf{ACA}_0$, the order $B_D(A)$ constructed as the union of the chain $\{B_{D, <n}(A):n\in\mathbb{N}\}$ satisfies the definition of $B_D(A)$, for any linear ordering $A$.
\endproof

\begin{definition}
Let $A$ and $A'$ be linear orders and $f: A\to A'$ be a morphism. We define
\[B_D(f): B_D(A)\to B_D(A')\] 
to be the unique map $g\colon B_D(A)\to B_D(B)$ such that:
\begin{enumerate}
\item $g(a^{\star}) = (f(a))^{\star}$, and
\item $g(\psi(t))=\psi(2^{D(g)}(t))$.
\end{enumerate}
\end{definition}

\begin{lemma}[$\aca$]
Suppose that $A$ and $A'$ are linear orders and $f:A \to A'$ is a morphism. Then, $B_D(f)$ exists.
\end{lemma}
\proof
This is similar to the previous proof. By induction, we define a sequence of embeddings 
$B_{D,n}(f)\colon B_{D,n}(A)\to B_{D,n}(B)$. Letting
\[B_{D,<n}(f)=\bigcup_{m<n} B_{D,m}(f)\colon B_{D,<n}(A)\to B_{D,<n}(B),\] 
we put:
\begin{enumerate}
\item $B_{D,n}(f)(a^{\star}) = (f(a))^{\star}$, and
\item $B_{D,n}(f)(\psi(t))=\psi(2^{D(B_{D,<n}(f))}(t))$.
\end{enumerate}
Clearly, $B_{D,<\omega}(f)$ satisfies the definition of $B_D(f)$. By induction on $n$, one shows that for  any $g\colon B_D(A)\to B_D(B)$ satisfying the definition of $B_D(f)$ we have 
\[g{\upharpoonright} B_{D,n}(A)=B_{D,n}(f).\] 
Thus indeed there is a unique mapping satisfying the definition of $B_D(f)$.
\endproof

%In general, by induction on the syntactic representation of $\alpha$, we set
%\[B_D(f)(\psi(2^{\alpha_1} + \dots + 2^{\alpha_k})) = \psi(2^{\alpha'_1} + \dots + 2^{\alpha'_k})\]
%so that $\alpha'_1 = D(f)(\alpha_1)$, and, writing 
%\[\alpha_i = t(x_1, \hdots, x_{k_1})\]
%as a $D$-term, we have
%\[\alpha'_i = t(B_D(f)(x_1), \hdots, B_D(f)(x_{k_i})),\]
%defined by induction on $i$.

\begin{lemma}[$\mathsf{ACA}_0$]\label{LemmaPullbacks}
Suppose $D$ is a $\subseteq$-preserving pre-dilator. Then, so too is $B_D$. 
\end{lemma}
\proof
The functoriality and $\subseteq$-preservation of $B_D$ are immediate from the definition. In order to check that $B_D$ is finitary it is enough to show that for any $A$ we have 
\[B_D(A)=\bigcup\limits_{A'\subseteq_{\mathsf{fin}}A}B_D(A').\] 
This is done by showing, via a straightforward induction on $n$, that
\[B_{D,n}(A)=\bigcup\limits_{A'\subseteq_{\mathsf{fin}}A}B_{D,n}(A').\] 
By a straightforward induction on $n$ we also show that for any order $A$ and any suborders $A',A''$ of $A$, we have 
\[B_{D,n}(A')\cap B_{D,n}(A'')=B_{D,n}(A'\cap A'').\] 
This implies that for any order $A$ and its suborders $A',A''$ we have 
\[B_{D}(A')\cap B_{D}(A'')=B_{D}(A'\cap A''),\]
i.e., that $B_D$ preserves pullbacks. This concludes the proof that $B_D$  is a $\subseteq$-preserving pre-dilator.
\endproof

For an arbitrary pre-dilator $D$,  we may define the pre-dilator $B_D$ to be $B_{D'}$, where $D'$ is some fixed $\subseteq$-preserving pre-dilator naturally isomorphic to $D$. 

\begin{lemma}[$\mathsf{ACA}_0$]\label{EquivalenceOfDefinitions} 
Suppose $D$ is a weakly finite dilator. Let $B_D$ be the functor defined in this section and $\hat B_D$ be the functor defined in Section \ref{SectFunctor}. Then, $B_D$ and $\hat B_D$ are naturally isomorphic.
\end{lemma}
\proof
Let $\theta\colon D\to D'$ be the natural isomorphism given by the definition of $B_D$ for arbitrary pre-dilators $D$.
For each $n\in\mathbb{N}$, let 
\[\eta_n: \hat B_D(n) \to B_D(n)\] 
be the function given by:
\begin{align*}
m &\mapsto m^\star, &\text{ if $m < n$,}\\
B_{\alpha_k}\dots B_{\alpha_1}(n) &\mapsto \psi(2^{\alpha_1'} + \dots + 2^{\alpha_k'}),\text{ where }\alpha_i'=\theta_{B_D(n)}(\alpha_i), &\text{ otherwise.}
\end{align*}
It is clear that the collection of $\eta_n$ for $n\in\mathbb{N}$ forms a natural isomorphism between $\hat B_D\colon \Nat\to \LO$ to $B_D\upharpoonright \Nat$, the restriction of $B_D$ to the category $\Nat$. 
Since both $\hat B_D$ and $B_D$ are finitary, this natural transformation extends to a natural isomorphism between the extension of $\hat B_D$ to $\LO$ and $B_D$, as desired.
\endproof

\begin{remark} We note, without giving details, that in fact the operation 
\[D\longmapsto B_D\]
 could be extended to a pullback-preserving functor on the category of pre-dilators; in other words, $B$ is a \emph{preptyx} and -- as we will see in the next section -- indeed it is a \emph{ptyx}. Furthermore, if we treat the category of pre-dilators as a category of structures (denotation systems can be regarded as structures), then in fact the definition of $B$ in terms of formal ``ordinal collapses'' will be a $\subseteq$-preserving functor. 
\end{remark}

\begin{remark} The construction of $B_D(A)$ in fact works for all $\subseteq$-preserving finitary functors $D\colon \LO\to \LO$. We decided not to cover the extension to finitary functors in this paper for the following reason: The fact that any dilator is naturally isomorphic to a $\subseteq$-preserving one is a known fact following from Girard's fundamental theorem of dilators. 
There is an unpublished result by the second author that establishes an analog of Girard's theorem for finitary functors on $\LO$ that in particularly implies that any finitary functor on $\LO$ is naturally isomorphic to a $\subseteq$-preserving one. 
This result shall appear in a forthcoming article by the second and third authors. 
\end{remark}

In order to make sense of $B_D$ in the case when $D$ is not $\subseteq$-preserving, it suffices to choose a denotation system $D'$ and put $B_{D}=B_{D'}$.

\section{Proof of the theorem}\label{SectTheorem}
In this section we prove the main theorem. We restate it, for convenience.

\begin{theorem}\label{MainRestated}
The following are equivalent over $\aca$:
\begin{enumerate}
\item \label{MainRestated1} $\Pi^1_1$-$\ca$;
\item \label{MainRestated4}for every weakly finite dilator $D$, $B_D$ is a weakly finite dilator;
\item \label{MainRestated3} for every dilator $D$, $B_D$ is a dilator.
\end{enumerate}
\end{theorem}
%The condition that $F$ be a ptyx, from the statement of Theorem \ref{TheoremMainIntroPtyx} on p. \pageref{TheoremMainIntroPtyx}, is, by definition, synonymous with \eqref{MainRestated3} of Theorem \ref{MainRestated}.
%As before we reason in $\aca$ unless otherwise stated.
\begin{lemma}[$\Pi^1_1$-$\ca$] \label{LemmaPi11CA}
For any dilator $D$ and any wellordering  $A$, $B_D(A)$ is wellordered.
\end{lemma}
\proof
By the definition of $B_D$, it suffices to consider the case that $D$ is $\subseteq$-preserving.

Let $\Omega$ be the largest wellordered initial segment of $B_D(A)$. This exists by \mbox{$\Pi^1_1$-$\ca$}.  By hypothesis, $D$ is a dilator, so $D(\Omega)$ is wellordered. 
By induction on ordinals 
$\alpha \in D(\Omega)$,  we show that for every $\psi(2^{\alpha_1} + \dots + 2^{\alpha_k}) \in \Omega$, if $\psi(2^{\alpha_1} + \dots + 2^{\alpha_k} + 2^\alpha) \in B_D(A)$, then 
\[a := \psi(2^{\alpha_1} + \dots + 2^{\alpha_k} + 2^\alpha) \in \Omega.\]
Suppose that $x <_{B_D(A)} a$. Then, either
\[x \leq \psi(2^{\alpha_1} + \dots + 2^{\alpha_k})\]
or else there are $l\in\mathbb{N}$ and $\beta_1, \hdots, \beta_l$ with $\beta_l <_{D(B_D(A))} \dots <_{D(B_D(A))} \beta_1 < \alpha$ such that
\[x \leq \psi(2^{\alpha_1} + \dots + 2^{\alpha_k} + 2^{\beta_1} + \dots + 2^{\beta_l}).\]
We show by induction on $i < l$ that 
\begin{enumerate}
\item $\beta_i \in D(\Omega)$,  and
\item $\psi(2^{\alpha_1} + \dots + 2^{\alpha_k} + 2^{\beta_1} + \dots + 2^{\beta_i}) \in \Omega$.
\end{enumerate}
By the definition of $B_D$,  we have
\[\beta_{i+1} \in D(B_D(A)\upharpoonright \psi(2^{\alpha_1} + \dots + 2^{\alpha_k} + 2^{\beta_1} + \dots + 2^{\beta_i})).\]
By the induction hypothesis on $i$,  $\psi(2^{\alpha_1} + \dots + 2^{\alpha_k} + 2^{\beta_1} + \dots + 2^{\beta_i}) \in \Omega$ (this is immediate by the assumption on $\alpha_1, \ldots, \alpha_k$ in the case $i =0$) and $\Omega$ is an initial segment of $B_D(A)$,  so the first claim follows.  The second follows from the induction hypothesis on $\alpha$.
We have shown that every element $x <_{B_D(A)} a$ belongs to $\Omega$ and thus
 $a \in \Omega$, as desired.

By a straightforward induction on $k$, it follows that
\[\psi(2^{\alpha_1} + \dots + 2^{\alpha_k}) \in \Omega\]
for all $\psi(2^{\alpha_1} + \dots + 2^{\alpha_k}) \in B_D(A)$, i.e., that $B_D(A) = \Omega$ and thus $B_D(A)$ is wellordered.
\endproof

\begin{lemma}[$\aca$]\label{LemmaWeaklyFinite}
Suppose that for every weakly finite dilator $D$, $B_D$ is a weakly finite dilator. Then, for every dilator $D$, $B_D$ is a dilator.
\end{lemma}
\proof
If $D$ and $D'$ are naturally isomorphic pre-dilators, then so too are $B_D$ and $B_{D'}$. Hence, it suffices to prove the theorem for denotation systems $D$. 
Let $D$ be a denotation system and enumerate all ordinal terms of $D$ by $d_0, d_1, \hdots$; write $n_i$ for the arity of $d_i$. We define a new dilator $\hat D$ consisting of terms $\hat d_i$ of arity $n_i+i$. Clearly, $\hat D$ will be weakly finite. The comparison rules for $
\hat D$ are given by
\[\hat d_i (x_1, \hdots, x_{n_i}, \hat x_1, \hdots, \hat x_i) <_{\hat D} \hat d_j (y_1, \hdots, y_{n_j}, \hat y_1, \hdots, \hat y_j)\]
if and only if one of the following holds:
\begin{enumerate}
\item $d_i(x_1,\hdots, x_{n_i}) <_{D} d_j(y_1,\hdots, y_{n_j})$; or else
\item $i = j, x_1 = y_1, \hdots, x_{n_i} = y_{n_i}$, and $(\hat x_1, \hdots, \hat x_i) <_{lex} (\hat y_1,\hdots, \hat y_j)$.
\end{enumerate}
It is easy to check that $<_{\hat D}$ is a linear order. And it is easy to see that for any $A$, we have an embedding of $\hat D (A)$ into $2^{A}\cdot D(A)$: 
$$\hat d_i(a_1,\ldots,a_{n_i}, b_1,\ldots,b_i)\longmapsto 2^A \cdot d_i(a_1,\ldots,a_{n_i}) + 2^{b_1}+\ldots+2^{b_i}.$$ Thus $\hat D$ is a dilator. Observe that for each finite order $A$, the order $\hat D(A)$ contains only terms of the form $\hat d_i(\ldots)$, where $i\le |A|$ and hence $\hat D(A)$ is finite. Therefore $\hat D$ is a weakly finite dilator. Thus $B_{\hat D}$ is a dilator.

In order to show that $B_D$ preserves wellfoundedness it is enough to find a strictly increasing map 
\[e\colon B_D(A)\to B_{\hat D}(\omega+A)\] 
for an arbitrary wellorder $A$. We fix a wellorder $A$ and define, by recursion on $n$, 
\[e_n: B_{D,n}(A) \to B_{\hat D,n}(\omega+A)\]
\begin{enumerate}
\item $e_n(a^\star)=(\omega+a)^\star$;
\item $e_n(\psi(2^{\alpha_1}+\ldots +2^{\alpha_k}))=\psi(2^{\alpha_1'}+\ldots +2^{\alpha_k'})$, where if $\alpha_i=d_j(a_{1},\ldots,a_{n_{j}})$, then $\alpha_i'=\hat d_{j}(e_{n-1}(a_{1}),\ldots,e_{n-1}(a_{n_j}), (j-1)^\star, \ldots, 0^\star)$.\end{enumerate}
By a straightforward induction on $n$ we show that $e_n$'s form a sequence of expanding strictly increasing maps. We finish the proof by defining $e=\bigcup\limits_{n<\omega} e_n$.
%\begin{enumerate}
%\item for $\alpha\in A$, we let $e\alpha = \alpha$;
%\item $e\psi(0) = \psi(0)$;
%\item $e\psi(2^{\alpha_1} + \dots + 2^{\alpha_k} + 2^{d_i(x_1, \hdots, x_{n_i})}) = \psi(2^{e\alpha_1} + \dots 2^{e\alpha_k} + 2^{\hat d_i(0, 1, \hdots, i-1, ex_1, \hdots, ex_{n_i})})$.
%\end{enumerate}
%It suffices to show that $e$ is order-preserving. Suppose we are given 
%\[\alpha = \psi(2^{\alpha_1} + \dots + 2^{\alpha_k}) <_{B_D(A)} \psi(2^{\beta_1} + \dots + 2^{\beta_l}) = \beta.\]
%By induction on $k$ we show that  $e\alpha<e\beta$. We can reduce to the case that $k = l$ and $\alpha_i = \beta_i$ for all $i < k$.
%Then, choosing $d_i$ and the constants $x_1, \hdots, x_{n_i}$ and $y_1, \hdots, y_{n_i}$ appropriately, we have 
%\begin{align*}
%e\psi(2^{\alpha_1} + \dots + 2^{\alpha_k}) 
%&= \psi(2^{e\alpha_1} + \dots + 2^{e\alpha_k})
%\end{align*}
%Writing 
%\begin{align*}
%\alpha_k &= d_i(x_1, \hdots, x_{n_i})\\
%\beta_k &= d_j(y_1, \hdots, y_{n_j}),
%\end{align*} 
%we have
%\begin{align*}
%e\psi(2^{\alpha_1} + \dots + 2^{\alpha_k}) 
%&= \psi(2^{e\alpha_1} + \dots + 2^{\hat d_i(0,1, \hdots, i-1, ex_1, \hdots, ex_{n_i})})\\
%&<\psi(2^{e\alpha_1} + \dots + 2^{\hat d_j(0,1, \hdots, j-1, ey_1, \hdots, ey_{n_j})})\\
%&= e\beta.
%\end{align*}
%Here, the crucial inequality follows by induction hypothesis.
%Since each $B_D(A)$ embeds into a wellorder, $B_D(A)$ is wellfounded for every wellorder $A$, as desired.
\endproof

In order to complete the proof, we have to show that if $B_D$ is a dilator for every dilator $D$, then $\Pi^1_1$-$\ca$ holds. This will be done by appealing to a theorem of Freund \cite{Fr19} whereby $\Pi^1_1$-$\ca$ is equivalent to a higher order fixed-point principle. 
Using the terminology of Freund \cite{Fr19}, this is the statement that every dilator $D$ has a wellfounded Bachmann-Howard fixed point (we recall this definition during the course of the forthcoming proof).

\begin{lemma}[$\aca$]\label{LemmaBachmannHoward}
Suppose that $B_D$ is a dilator for every dilator $D$. Then, every dilator has a wellfounded Bachmann-Howard fixed point.
\end{lemma}
\proof
It suffices to restrict to the case where $D$ is a denotation system and, in particular, $\subseteq$-preserving.
We need to find a wellordering $A$ and a function 
\[\theta:D(A)\to A\]
such that the following hold:
\begin{enumerate}
\item \label{PropBHF1} whenever we have $\alpha<_{D(A)}\beta$ and $\supp(\alpha)<_A \theta(\beta)$, then we have \mbox{$\theta(\alpha) <_A \theta(\beta)$}; and
\item \label{PropBHF2}  $\supp(\alpha) <_A \theta(\alpha)$ for every $\alpha \in D(A)$.
\end{enumerate}

Consider the dilator $F=(\omega+1)D +\omega$. This is defined as follows: Given a linear order $B$, the order $F(B)$ consists of terms of two types:
\begin{enumerate}
\item $(\omega+1)\alpha+a$, where $\alpha\in D(B)$ and $a\le \omega$;
\item $\Lambda+n$, where $n<\omega$.
\end{enumerate}
The terms are compared according to the following rules:
\begin{enumerate}
\item terms of the first type are always smaller than the terms of the second type;
\item $(\omega+1)\alpha+a$ is smaller than $(\omega+1)\beta+b$ if and only if either $\alpha<\beta$ or $\alpha=\beta$ and $a<b$;
\item $\Lambda+n$ is smaller than $\Lambda+m$ if and only if $n<m$.
\end{enumerate}
For a strictly increasing function $f\colon B\to C$ we put 
\begin{align*}
F(f):F(B) &\to F(C)\\ 
(\omega+1)\alpha+a &\mapsto (\omega+1)(D(f)(\alpha))+a;\\
\Lambda + n &\mapsto \Lambda + n.
\end{align*}
It is easy to see that $F$ is indeed a dilator.

We put $A=B_{F}(0)$. Note that all elements of $A$ are of the form $\psi(2^{\alpha_1}+\ldots+2^{\alpha_n})$, where $\alpha_i\in F(A)$. We define 
\[\theta\colon D(A)\to A\] 
to be the function that maps $\alpha\in D(A)$ to the least element of $A$ of the form 
\[\psi(2^{\alpha_1}+\ldots+2^{\alpha_n}+2^{(\omega+1)\alpha+\omega}).\] 
This function is well defined, i.e., for any $\alpha\in D(A)$ there is some element of $A$ as above.
Indeed, the elements of the form $\psi(2^{\Lambda+n})$ are cofinal in $A$. 
Thus, for any $\alpha\in D(A)$ one can find $n$ large enough so that $\alpha\in D(A{\upharpoonright}\psi(2^{\Lambda+n}))$  and hence 
\[\psi(2^{\Lambda+n}+2^{(\omega+1)\alpha+\omega})\in A.\]

Let us verify that $\theta$ satisfies property \eqref{PropBHF1}: we suppose $\alpha<_{D(A)}\beta$ and $\supp(\alpha)<_A \theta(\beta)$ and claim that $\theta(\alpha) <_A \theta(\beta)$. Let $\beta_1$, $\hdots$, $\beta_n$ be such that
\[\theta(\beta) = \psi(2^{\beta_1}+\ldots+2^{\beta_n}+2^{(\omega+1)\beta+\omega}).\] 
Observe that all elements of the form 
\[\psi(2^{\beta_1}+\ldots+2^{\beta_n}+2^{(\omega+1)\beta+m})\] 
are in $A$ and are cofinal below $\theta(\beta)$. It follows that, since 
\[\supp(\alpha) = \supp((\omega+1)\alpha+\omega)\]
is a finite set, we can find $m\in\mathbb{N}$ large enough so that 
\[\supp(\alpha)<_A\psi(2^{\beta_1}+\ldots+2^{\beta_n}+2^{(\omega+1)\beta+m}).\] 
For such an $m$, $\psi(2^{\beta_1}+\ldots+2^{\beta_n}+2^{(\omega+1)\beta+m}+2^{(\omega+1)\alpha+\omega})$ is in $A$. Therefore
\begin{align*}
\theta(\alpha)
&\le \psi(2^{\beta_1}+\ldots+2^{\beta_n}+2^{(\omega+1)\beta+m}+2^{(\omega+1)\alpha+\omega})\\
&<\psi(2^{\beta_1}+\ldots+2^{\beta_n}+2^{(\omega+1)\beta+\omega})\\
&=\theta(\beta).
\end{align*}

The fact that $\theta$ satisfies \eqref{PropBHF2} is immediate from the construction. Indeed, for any $\alpha\in D(A)$, the value $\theta(\alpha)$ is of the form $\psi(2^{\alpha_1}+\ldots+2^{\alpha_n}+2^{(\omega+1)\alpha+\omega})\in A$ and we have $$\supp(\alpha)=\supp((\omega+1)\alpha+\omega)<_A\psi(2^{\alpha_1}+\ldots+2^{\alpha_n})<_A\theta(\alpha),$$
as desired.
\endproof

Putting together the lemmata in this section, the proof of Theorem \ref{MainRestated} is now complete. 

\begin{remark}
By appealing to the main theorem of Freund \cite{Fr20} and carefully formalizing the definition of $B_D$ in Section \ref{SectCollapse}, the equivalence in Theorem \ref{MainRestated} can be proved in $\rca$.
\end{remark}

\bibliographystyle{abbrv}
\bibliography{References}

\begin{thebibliography}{10}

\bibitem{AfRa09}
B.~Afshari and M.~Rathjen.
\newblock {Reverse mathematics and well-ordering principles: a pilot study}.
\newblock {\em Ann. Pure Appl. Logic}, 160:231--237, 2009.

\bibitem{AFRW}
J.~P. Aguilera, A.~Freund, M.~Rathjen, and A.~Weiermann.
\newblock {Ackermann and Goodstein Go Functorial}.
\newblock {\em Pacific J. Math.}
\newblock In press.

\bibitem{Ba50}
H.~Bachmann.
\newblock Die normalfunktionen und das problem der ausgezeichneten folgen von
  ordnungszahlen.
\newblock {\em Vierteljschr. Naturforsch. Ges. Z{\"u}rich}, 95:114--147, 1950.

\bibitem{Bu87}
W.~Buchholz.
\newblock {An independence result for $\Pi^1_1$-CA+BI}.
\newblock {\em Ann. Pure Appl. Logic}, 33(2):131--155, 1987.

\bibitem{BCW94}
W.~Buchholz, A.~Chicon, and A.~Weiermann.
\newblock {A uniform approach to fundamental sequences and hierarchies}.
\newblock {\em Math. Log. Q.}, 40:273--286, 1994.

\bibitem{Fe72}
S.~Feferman.
\newblock Infinitary properties, local functors, and systems of ordinal
  functions.
\newblock pages 63--97, 1972.

\bibitem{Fr19}
A.~Freund.
\newblock {$\Pi^1_1$-comprehension as a well-ordering principle}.
\newblock {\em Adv. Math.}, 355, 2019.

\bibitem{Fr20}
A.~Freund.
\newblock {Computable aspects of the Bachmann-Howard principle}.
\newblock {\em J. Math. Log.}, 355, 2020.

\bibitem{GiInf}
J.-Y. Girard.
\newblock {\em Proof Theory and logical complexity, II}.
\newblock Book Draft.

\bibitem{Gi81}
J.-Y. Girard.
\newblock {$\Pi^1_2$}-logic, part 1: Dilators.
\newblock {\em Ann. Math. Logic}, 21(2):75 -- 219, 1981.

\bibitem{Gi87}
J.-Y. Girard.
\newblock {\em Proof Theory and logical complexity, I}.
\newblock 1987.

\bibitem{Hi94}
J.~Hirst.
\newblock {Reverse mathematics and ordinal exponentiation}.
\newblock {\em Ann. Pure Appl. Logic}, 66:1--8, 1994.

\bibitem{Kr52}
G.~Kreisel.
\newblock {On the Interpretation of Non-Finitist Proofs, II}.
\newblock {\em J. Symbolic Logic}, 17:43--58, 1952.

\bibitem{MaMo11}
A.~Marcone and A.~Montalb{\'a}n.
\newblock {The Veblen Function for Computability Theorists}.
\newblock {\em J. Symbolic Logic}, 76:576--602, 2011.

\bibitem{Ra99}
M.~Rathjen.
\newblock {Explicit Mathematics with the Monotone Fixed Point Principle. II:
  Models}.
\newblock {\em J. Symbolic Logic}, 64:517--550, 1999.

\bibitem{RaVV15}
M.~Rathjen and P.~F.~V. Vizca\'ino.
\newblock {Well ordering principles and bar induction}.
\newblock In R.~Kahle and M.~Rathjen, editors, {\em Gentzen's Centenary: The
  quest for consistency}, pages 533--561. 2015.

\bibitem{RaWe11}
M.~Rathjen and A.~Weiermann.
\newblock {Reverse Mathematics and Well-Ordering Principles}.
\newblock In S.~B. Cooper and A.~Sorbi, editors, {\em Computability in Context:
  Computation and Logic in the Real World}. 2011.

\bibitem{ThRa20}
I.~A. Thomson and M.~Rathjen.
\newblock {Well ordering principles and bar induction}.
\newblock In R.~Kahle and M.~Rathjen, editors, {\em The legacy of Kurt
  Sch\"utte}. 2020.

\bibitem{Wa99}
S.~Wainer.
\newblock {Accessible Recursive Functions}.
\newblock {\em Bull. Symbolic Logic}, 5:367--388, 1999.

\end{thebibliography}

\end{document}